\documentclass[final,11pt]{article}
\usepackage{authblk}
\usepackage[letterpaper]{geometry}
\usepackage[utf8]{inputenc}
\usepackage[english]{babel}

\usepackage{amssymb,amsthm,amscd,latexsym,mathrsfs,units,enumerate,bm,bbm,cancel,physics,mathtools,braket}
\usepackage{enumitem}
\usepackage{color,xcolor}
\definecolor{gogreen}{rgb}{0.05,0.45,.1}
\usepackage{graphicx,wrapfig}

\usepackage[colorlinks]{hyperref}
\usepackage[notref, notcite]{showkeys}

\renewcommand{\geq}{\geqslant}
\renewcommand{\leq}{\leqslant}

\newcommand{\be}{\begin{equation}}
\newcommand{\ee}{\end{equation}}
\newcommand\ba{\begin{equation}\begin{aligned}}
\newcommand\ea{\end{aligned}\end{equation}}
\theoremstyle{definition} 

\newcommand{\LL}{L^2(\mathbb{R}^2)}

\newcommand{\LLO}{L^2(\Omega)}

\newcommand{\lv}{\left\Vert}
\newcommand{\rv}{\right\Vert}

\newcommand{\half}{\frac{1}{2}}
\newcommand{\R}{\mathbb{R}}

\newcommand{\eps}{\varepsilon}

\newcommand{\sgn}{\textrm{ sgn}}
\newcommand{\OM}{\mathbb{R}\times(-L,L)}

\theoremstyle{plain}
\newtheorem{theorem}{Theorem}[section]

\newtheorem{proposition}[theorem]{Proposition}

\newtheorem*{thm*}{Theorem}

\theoremstyle{definition} 

\newtheorem{hyp}{Hypothesis}

\theoremstyle{remark}
\newtheorem*{remark*}{Remark}

\title{Global asymptotic stability of KdV--Burgers fronts \protect\\ 
in a weakly two-dimensional model}

\author[1]{Jared~C.~Bronski\thanks{E-mail:bronski@illinois.edu}}
\author[1]{Olivia Clifton\thanks{E-mail:ocannon2@illinois.edu}}
\author[1]{Vera~Mikyoung~Hur\thanks{E-mail:verahur@illinois.edu}}
\affil[1]{Department of Mathematics, University of Illinois Urbana-Champaign, \protect\\
Urbana, IL 61801, USA}

\begin{document}

\maketitle

\begin{abstract}
We study front-type solutions of nonlinear dispersive-dissipative PDEs modeling the propagation of undular bores in a channel in two dimensions. The system extends the Korteweg--de Vries--Burgers (KdVB) equation by incorporating weak transverse motion, and admits the KdVB fronts as one-dimensional solutions. We investigate their stability under general two-dimensional perturbations. We prove that a one-dimensional front is a global asymptotic attractor in the weakly two-dimensional setting when the channel is sufficiently narrow in the transverse direction and the relative dispersion parameter lies in a range for stability to one-dimensional perturbations. Particularly, the front is unique up to spatial translations. The proof extends the energy method for temporally-modulated perturbed solutions, developed previously in the one-dimensional setting, to accommodate the transverse dynamics.
\end{abstract}

\section{Introduction}\label{sec:intro}
We study front-type solutions of the following coupled nonlinear dispersive-dissipative equations in two spatial dimensions, modeling the propagation of undular bores:
\begin{equation}\label{e:Rajopadhye}
\left\{\begin{aligned}
&u_t  + u u_x  + v_y = u_{xx} + u_{yy} + \nu u_{xxx}, \\
& v_t + u_y = v_{xx}+v_{yy}.
\end{aligned}\right.
\end{equation}
Here $t\geq0$ denotes time, and $x,y\in\mathbb{R}$ are the spatial variables; $x$ denotes the direction of wave propagation and $y$ the transverse direction; $u=u(x,y,t)$ represents the displacement of the fluid surface, and $v=v(x,y,t)$ the transverse velocity; $\nu\in\mathbb{R}$ measures the strength of dispersion relative to dissipation. 

The Korteweg--de Vries--Burgurs (KdVB) equation, after normalization of parameters, takes the form
\begin{equation}\label{e:KdVB}
u_t + u u_x = u_{xx} + \nu u_{xxx},
\end{equation}
which has long served as a fundamental mathematical model for the propagation of undular bores on a one-dimensional fluid surface, where the motion is independent of the transverse direction. See, for instance, \cite{BARKER2025} and references therein. 
One may view \eqref{e:Rajopadhye} as a weakly two-dimensional extension of \eqref{e:KdVB} that incorporates transverse motion. More specifically, the first equation in \eqref{e:Rajopadhye} is the KdVB equation for $u$, coupled to $v$ through the term $v_y$. We emphasize that \eqref{e:Rajopadhye} is a slight modification of the weakly two-dimensional model proposed by Rajopadhye \cite{Raj1996}:
\[
\left\{\begin{aligned}
&u_t  +u_x + u u_x  + v_y = u_{xx} + u_{yy} + \nu u_{xxx}, \\
&v_t + u_y = 0.
\end{aligned}\right.
\]
See also references therein. The advection term $u_x$ in the first equation can be removed by the simple change of variables $u \mapsto 1+u$ and is therefore omitted from \eqref{e:Rajopadhye}. Although the viscosity terms $v_{xx}+v_{yy}$ in the second equation of \eqref{e:Rajopadhye} do not appear in Rajopadhye's formulation, their inclusion is consistent with the spirit of the original derivation in \cite{Raj1996} and is convenient for our purposes.    

We pose \eqref{e:Rajopadhye} on the strip $\{(x,y)\in \mathbb{R}^2: |y|<L\}$ for some $L>0$, subject to the boundary conditions: 
\begin{equation}\label{e:BC}
\left\{\begin{aligned}
& u_y(x,-L,t)=0, \quad& &u_y(x,L,t)=0,\\
& v(x,-L,t)=0, & &v(x,L,t)=0.
\end{aligned}\right.
\end{equation}
That is, $u$ satisfies homogeneous Neumann boundary conditions while $v$ satisfies homogeneous Dirichlet boundary conditions in the transverse direction. Since undular bores are commonly observed in rivers and channels, it makes sense to work on strip-like domains with finite transverse extent. As we shall see, the strip width $2L$ plays an important role in our analysis. 
We supplement \eqref{e:Rajopadhye} and \eqref{e:BC} with the initial conditions: 
\begin{equation}\label{e:IC}
u(x,y,0)=u_0(x,y)\quad\text{and}\quad v(x,y,0)=v_0(x,y).
\end{equation}

We assume that \eqref{e:Rajopadhye}, \eqref{e:BC}, and \eqref{e:IC} are well-posed in the $H^3$ space setting. Rajopadhye \cite{Raj1996} established well-posedness for the equations considered there. We expect that the same line of argument can be adapted to accommodate the additional viscosity terms in the second equation of \eqref{e:Rajopadhye} and the strip-like domain. We omit the details here. 

Clearly, if $u(x,t)$ is a solution of \eqref{e:KdVB}, then $(u(x,y,t),v(x,y,t)) = (u(x,t),0)$ is a solution of \eqref{e:Rajopadhye} and \eqref{e:BC}. 
Notably, \eqref{e:KdVB} admits traveling front solutions of the form $u(x,t)=\phi(x-ct)$, where 
\[
-c\phi_x+\phi\phi_x=\phi_{xx}+\nu \phi_{xxx}, \qquad \phi(x)\to \phi_\pm \quad\text{as $x\to \pm \infty$},
\]
for $\phi_\pm$ satisfying $\phi_->\phi_+$ and $c=\frac{\phi_- + \phi_+}{2}$. Accordingly, $(u(x,y,t),v(x,y,t)) = (\phi(x-ct),0)$ makes a one-dimensional solution of \eqref{e:Rajopadhye} and \eqref{e:BC}. 
The traveling front solutions of \eqref{e:KdVB} considered in \cite{BARKER2025} satisfy $\phi_\pm=\mp1$ so that $c=0$, and the fronts are stationary. It is elementary that \eqref{e:KdVB} can then be transformed by a suitable change of variables to admit traveling front solutions for arbitrary $\phi_->\phi_+$ with $c=\frac{\phi_- + \phi_+}{2}$. It therefore suffices to consider 
\begin{equation}\label{e:KdVB0}
\frac12(\phi^2-1)=\phi_x+\nu \phi_{xx}, \qquad \phi(x)\to \mp1 \quad\text{as $x\to \pm \infty$}.
\end{equation}

The existence of stationary front solutions of \eqref{e:KdVB}, satisfying \eqref{e:KdVB0}, has been well established. See, for instance, \cite{BS1985}. We also recall that the profile is monotone if and only if $|\nu|\leq 1/4$. Their global asymptotic stability to general one-dimensional perturbations has recently been established for $|\nu| \in [0,1/4]\bigcup [0.2533,3.9]$, combining rigorous analysis for monotone profiles with validated numerics for non-monotone profiles \cite{BARKER2025}. See also \cite{chen2026} for a more recent ``pen-and-paper'' proof for $|\nu|\in[1/4,1/2]$, and \cite{KangVasseur2017} for a similar result for a more general convex flux in the absence of dispersion. Together, stability under one-dimensional perturbations is known for $|\nu|\leq 3.9$. Our objective here is to assess whether these one-dimensional fronts remain stable when regarded as solutions of \eqref{e:Rajopadhye} subject to arbitrary bounded two-dimensional perturbations.

Transverse dynamics of coherent structures in higher dimensions have been studied for a variety of inviscid fluid models. 
For example, stability and transverse instability have recently been studied extensively for front and roll-wave solutions of the inviscid St. Venant equations in both strip-like domains and the whole plane \cite{yang2025multidimensional}. When the Froude number satisfies $F<2$, hydraulic shocks are universally stable in both settings. When $F>2$, on the other hand, the stability picture is considerably richer: the parameter space contains regions of stability and transverse instability, the latter giving rise to two-dimensional herringbone patterns. Transverse stability and instability have also been established for roll and solitary waves in planar extensions of the Korteweg--de Vries equation, most notably the Kadomtsev--Petviashvili equation and its generalizations. See, for instance, \cite{bhavna2022transverse,haragus2017transverse,johnson2010transverse,spektor1988transverse}. 

Significant recent progress has likewise been made on the stability of multidimensional shock waves in viscous fluid models without dispersion. One effective strategy is to establish an $L^2$ contraction by combining a Poincar{\'e}-type inequality with a change of variables that maps the coordinate in the direction of wave propagation onto a bounded interval \cite{kang2025multiDburgers,wangNSviscous2024}. Such a change of variables is possible for viscous conservation laws whose shock profile is monotone in the direction of propagation. More recently, this approach has been extended to include dispersion in the one-dimensional case, yielding stability of the front solutions of \eqref{e:KdVB} for $|\nu|\in[1/4,1/2]$, which includes non-monotone profiles \cite{chen2026}. To overcome the non-monotonicity of the profile, \cite{chen2026} constructed a rather delicate bound on the front derivative $\phi_x$ in terms of simple algebraic functions. This allows them to make the desired change of variables and establish a Poincar{\'e}-type inequality on each interval over which the profile is monotone. These local estimates are then assembled to obtain a global bound. 

Despite these developments, there are comparatively few results on the stability and instability of higher-dimensional fronts or shock waves when nonlinearity, dissipation, and dispersion are all accounted for. The present work addresses this gap for \eqref{e:Rajopadhye}.

In Section \ref{s:stab}, we prove that $(\phi,0)$, a one-dimensional front solution of \eqref{e:Rajopadhye} and \eqref{e:BC}, is globally asymptotically stable in the $L^p$ space setting for $p\in(2,\infty)$ if the strip width is sufficiently small, assuming that the one-dimensional Schr\"odinger operator $-\frac{d^2}{dx^2}+\frac12\phi_x(x)$ has exactly one negative eigenvalue, denoted by $\lambda_{\min}$. More specifically, $L<L_{\max}:=\frac{\pi}{2\sqrt{-\lambda_{\min}}}$. Numerical computations indicate $2L_{\max}\approx10$, with a mild but nontrivial dependence on $\nu$. As in \cite{BARKER2025}, we impose no smallness assumption on the initial perturbation. It is in this sense that our stability result is global. Furthermore, our result implies uniqueness of $(\phi,0)$ up to translations in the $x$-direction. In Section \ref{s:decay}, under a slightly stronger localization assumption on the initial condition in the $x$-variable, we extend the stability result to $p\in(1,2]$ and derive algebraic decay rates in time.  

\paragraph{Energy method for global asymptotic stability.} 
The proof of Theorem \ref{t:p>2} is based on the energy method for the perturbation of the stationary front $(\phi,0)$ in $L^2(\R\times[-L,L])\times L^2(\R\times[-L,L])$. We introduce a time-dependent modulation of the front, denoted by $x_0(t)$, and write 
\[
u(x,y,t)=\phi(x-x_0(t))+w(x,y,t),
\]
where 
\[
\frac{dx_0}{dt} = -\gamma \int_{-L}^L\int_{-\infty}^\infty \phi_x(x-x_0(t))w(x,y,t)~dxdy
\]
for some constant $\gamma>0$ to be specified during the course of the proof. This equation for $x_0$ can be interpreted as some gradient flow dynamics for the squared $L^2$ distance between the front and the perturbed solution. Alternatively put, at each time, $x_0(t)$ selects a translate of the front that ``best'' tracks the perturbed solution. 

The $L^2$ energy of the perturbation then satisfies
\[ 
\half\frac{d}{dt}(\lv w \rv_{L^2} ^ {2} + \lv v \rv_{L^2}^{2}) =  \iint - \varepsilon(|\nabla w| ^ { 2 } + |\nabla v| ^ { 2 }) - \langle (w,v) , \mathcal{H}(w,v) \rangle
\]
for $\varepsilon>0$ sufficiently small (see \eqref{e:step1}), where 
\[ 
\mathcal{H}(w,v) = \Big(-(1-\varepsilon)\Delta w + \frac{1}{2}\phi_x(x)w + \gamma \mathcal{P}_{\phi_x} w, -(1-\varepsilon)\Delta v\Big). 
\] 
Here $\mathcal{P}_{\phi_x}$ denotes the rank-one projection operator onto the span of $\phi_x$ in $L^2(\mathbb{R}\times [-L,L])$. Importantly, $\lv w(t)\rv_{L^2}^2+\lv v(t)\rv_{L^2}^2$ is non-increasing provided that $\mathcal{H}$ is positive semi-definite, that is, $\langle(w,v),\mathcal{H}(w,v)\rangle\geq0$.

The idea of time-dependent modulation was recently introduced in \cite{BARKER2025} for \eqref{e:KdVB}, where the corresponding $L^2$ energy of a one-dimensional perturbation is shown to be non-increasing in time provided that 
\[
\mathcal{H}^{(1)}:=-\frac{d^2}{dx^2}+ \frac { 1 } { 2 } \phi _ { x }(x) + \gamma \mathcal{P}^{(1)}_{\phi_x},
\]
the one-dimensional counterpart of $\mathcal{H}$, is positive semi-definite. Here $\mathcal{P}^{(1)}_{\phi_x}$ denotes the projection operator onto $\phi_x$ in $L^2(\mathbb{R})$. It is elementary that $-\frac{d^2}{dx^2}+\frac12\phi_x(x)$, without the rank-one perturbation by $\mathcal{P}^{(1)}_{\phi_x}$, has at least one negative eigenvalue because $\int\phi_x(x)~dx<0$. It was then shown in \cite{BARKER2025} that $\mathcal{H}^{(1)}$ becomes positive semi-definite for $\gamma>0$ sufficiently large if $-\frac{d^2}{dx^2}+\frac12\phi_x(x)$ has exactly one negative eigenvalue. This ultimately yields global asymptotic stability for \eqref{e:KdVB}. Additionally, the spectral condition was verified for $|\nu| \in [0, 1/4]$ through rigorous analysis as well as for $|\nu|\in [0.2533, 3.9]$ through validated numerics. 

We show that $\mathcal{H}$ is positive semi-definite if the strip width $2L$ is sufficiently small, for $\gamma>0$ sufficiently large, using the spectral information of $\mathcal{H}^{(1)}$, as long as $-\frac{d^2}{dx^2}+\frac12\phi_x(x)$ has exactly one negative eigenvalue. Consequently, $\lv w(t)\rv_{L^2}^2+\lv v(t)\rv_{L^2}^2$ is non-increasing. More precisely, 
\[
L < L_{\text{\rm max}} := \frac{\pi}{2\sqrt{-\lambda_{\min}}},
\]
where $\lambda_{\min}$ is the unique negative eigenvalue of $-\frac{d^2}{dx^2} + \frac{1}{2}\phi_x(x)$. 

 The restriction on $L$ appears to be a limitation of our argument rather than a sharp stability threshold. Indeed, numerical experiments suggest that one-dimensional fronts remain stable in the weakly two-dimensional setting for strip widths substantially larger than those covered by our theorem. This is consistent with Rajopadhye's stability result for monotone fronts under small perturbations on the whole plane $\mathbb{R}^2$ \cite{Raj1996}, which suggests that monotone profiles may remain stable for arbitrary $L$.  

Before the time-dependent modulation was introduced in \cite{BARKER2025}, stability arguments for fronts in related models typically formulated an energy estimate for an antiderivative of the perturbation rather than for the perturbation itself. For example, Pego \cite{Pego} developed this approach for \eqref{e:KdVB}, showing that the antiderivative of the perturbation, denoted by $W$, satisfies an energy identity of the form
\[
\half \frac{d}{dt} \| W \|_{L^2}^2 = -\int W_x^2 + \half \int \phi_xW^2 - \int W^2W_x.
\]
When the profile is monotone, the contribution $\half\int\phi_xW^2$ is non-positive. By showing that the remaining cubic term can be controlled if $\|W\|_{L^\infty}$ is sufficiently small, \cite{Pego} proved stability of monotone profiles under sufficiently small, mean-zero perturbations. Related arguments employing antiderivatives, all of which essentially rely on the monotonicity of the profile, appeared in \cite{ACH2014,Engler,goodman1989stability,Wang}, addressing stability and related questions. Khodja \cite{Khodja} and Naumkin and Shishmarev \cite{naumkinShishmarev} extended this framework perturbatively to weakly non-monotone profiles. Particularly, \cite{naumkinShishmarev} introduced a modified weighted energy for an antiderivative of the perturbation for \eqref{e:KdVB}, showing that
\[ 
\frac{d}{dt} ( \lv W \rv_{L^2}^2 + \varepsilon \lv x^2 W \rv_{L^2}^2  )  \leq 0
\]
for $\varepsilon>0$ sufficiently small. They further showed that $\lv W(t) \rv_{L^2}^2 + \varepsilon \lv x^2 W(t) \rv_{L^2}^2 \to 0 $ as $t\to\infty$ for $|\nu|$ sufficiently close to the monotonicity threshold. Rajopadhye \cite{Raj1996} took Pego's antiderivative approach to prove the stability of one-dimensional monotone fronts under small perturbations in a weakly two-dimensional model posed on $\mathbb{R}^2$.

The strategy introduced in \cite{BARKER2025} and adopted here differs fundamentally in that it works directly with the perturbation of a time-modulated front. This approach has two main advantages. First, it is not perturbative in the parameter $\nu$, and therefore applies over a substantially larger parameter range that includes genuinely oscillatory profiles. Second, it imposes no smallness assumption on the initial perturbation, yielding a global stability result. This also implies uniqueness up to translation in the relevant function spaces. We note that recent $L^2$ contraction arguments for multidimensional viscous shocks \cite{kang2025multiDburgers,wangNSviscous2024} and for one-dimensional dissipative-dispersive fronts \cite{chen2026} similarly introduce a time-dependent translation in order to work directly with the perturbation rather than its antiderivative. 

\section{Global asymptotic stability in \texorpdfstring{$L^p$}{} for \texorpdfstring{$p \in (2,\infty)$}{}}\label{s:stab}

We establish the global asymptotic stability of a KdVB front within \eqref{e:Rajopadhye} and \eqref{e:BC} under the $L^p$ norms for $p \in (2,\infty)$. 

We begin by stating the well-posedness hypothesis and the main stability theorem.

\begin{hyp}[Well-posedness]\label{hyp:1}

Let $(u, v)$ denote a solution of \eqref{e:Rajopadhye} and \eqref{e:BC}. We assume that 
 \begin{align*}
 &w(x,y, t) := u(x,y,t) - \phi(x- x_0(t)) \in H^3(\R \times [-L,L]),\\
 &v(x,y, t) \in H^3(\R \times [-L,L]), 
 \end{align*}
for $t\in(0,\infty)$, where $\phi$ is a solution to \eqref{e:KdVB0} and $x_0$ is as in \eqref{e:xi_0} below.
\end{hyp}

\begin{theorem}[Global asymptotic stability]\label{t:p>2}
Let $(u, v)$ denote a solution of \eqref{e:Rajopadhye} and \eqref{e:BC}, and $\phi$ a solution of \eqref{e:KdVB0}. Suppose that $w(x,y,t):=u(x,y,t) - \phi(x - x_0(t))$ and $v(x,y,t)$ satisfy Hypothesis \ref{hyp:1}, where 
\begin{equation}\label{e:xi_0}
\frac{dx_0}{dt} = -\gamma \iint_{\R\times[-L,L]} \phi_x(x-x_0(t))w(x,y,t)~dxdy 
\end{equation}
for some constant $\gamma>0$ to be be specified.
Suppose 
\begin{equation}\label{e:stability}
\text{$-(1-\epsilon)\frac{d^2}{dx^2}+\frac12\phi_x(x)$ has exactly one negative eigenvalue}
\end{equation}
for some $\epsilon>0$ sufficiently small, denoted by $\lambda_{\min}$. Let $L < L_{\max} := \frac{\pi}{2\sqrt{|\lambda_{\min}|}}$.  Then 
\[
\|w(t)\|_{L^p} + \| v(t)\|_{L^p} \to 0 \quad \text{as $t \to \infty$} \quad \text{for $p\in(2,\infty)$}.
\]
\end{theorem}

It is noteworthy that \eqref{e:stability} holds for $|\nu| \in [0, 1/4] \cup [0.2533, 3.9]$ \cite{BARKER2025}. 

We divide the proof of Theorem \ref{t:p>2} into three steps:
\begin{itemize}
    \item[$\bullet$] Step 1: $\displaystyle \int_0^\infty ( \lv \nabla w(t) \rv_{L^2}^2 + \lv \nabla v(t) \rv_{L^2}^2 )~dt<\infty$. 
    \item[$\bullet$] Step 2: $  \lv \nabla w(t) \rv_{L^2}^2 + \lv \nabla v(t) \rv_{L^2}^2 \to 0$ as $t\to \infty.$
    \item[$\bullet$] Step 3: $\lv w(t)\rv_{L^p} + \lv v(t)\rv_{L^p} \to 0$ as $t \to \infty$ for $p\in(2,\infty)$.
\end{itemize}

Step 1 follows from an energy estimate 
\[
\frac{d}{d t}(\lv w \rv_{L^2}^2 + \lv v\rv_{L^2}^2) \leq -C (\lv \nabla w \rv_{L^2}^2 + \lv \nabla v \rv_{L^2}^2)
\]
for some constant $C>0$. This implies the boundedness of $\lv w(t) \rv_{L^2}$ and $\lv v(t) \rv_{L^2}$, together with the integrability of $\lv \nabla w(t) \rv_{L^2}^2 + \lv \nabla v(t) \rv_{L^2}^2 $ for $t\in(0,\infty)$. The latter is \emph{almost} sufficient to show that $\lv \nabla w(t) \rv_{L^2}^2$, $\lv \nabla v(t) \rv_{L^2}^2 \to 0$ as $t\to\infty$, but not quite. One must exclude, for example, infinitely many $O(1)$ excursions from $0$ whose durations shorten sufficiently rapidly. 

Step 2 rules out such behavior by constructing a differential inequality for $\zeta(t):=\lv \nabla w(t) \rv_{L^2}^2 + \lv \nabla v(t) \rv_{L^2}^2 $. More precisely, we show that $\frac{d\zeta}{dt}$ is majorized by a finite sum of powers of $\zeta$. This prevents $\zeta$ from increasing arbitrarily rapidly once it becomes small, so that $\zeta(t) \to 0$  as $t\to \infty$. 

Step 3 focuses on turning the decay of $\lv \nabla w(t) \rv_{L^2}^2 + \lv \nabla v(t) \rv_{L^2}^2 $, along with the boundedness of $\lv w(t) \rv_{L^2}$ and $\lv v(t) \rv_{L^2}$, into a decay of $\lv w(t) \rv_{L^p}$ and $\lv v(t) \rv_{L^p}$. 
For $v$, which satisfies homogeneous Dirichlet boundary conditions in the $y$ variable, this follows immediately from the Poincar{\'e} inequality. On the other hand, for $w$, satisfying homogeneous Neumann boundary conditions in $y$, additional care is required. The standard Gagliardo--Nirenberg inequality would lead to an estimate of the form 
\[
\lv w \rv_{L^p} \leq C\lv w \rv_{L^2}^\theta \lv \nabla w \rv_{L^2}^{1-\theta},
\]
which would give the desired decay of the $L^p$ norm. Unfortunately, the standard Gagliardo--Nirenberg inequality does not directly apply because $w$ is defined only on a strip.  
Extending $w$ by zero across $y=\pm L$ does not produce an $H^1(\mathbb{R}^2)$ function because the boundary values need not vanish. A smooth extension to $\mathbb{R}^2$ must instead taper the boundary values to zero over a larger transverse region. This introduces an error term in the Gagliardo--Nirenberg inequality proportional to $\lv w \rv_{L^2}$. 

In Proposition \ref{prop:G-N_strip}, we carry out such a construction explicitly and derive a version of the Gagliardo--Nirenberg inequality adapted to a strip. By taking the width $R$ of the transverse extension region sufficiently large and choosing the exponent $\theta$ in the Gagliardo--Nirenberg inequality appropriately, the coefficient of the additional $\lv w\rv_{L^2}$ term can be made arbitrarily small, although it cannot be made zero without causing the coefficient of the $\|\nabla w\|_{L^2}$ term to diverge. 
On the other hand, since $\lv\nabla w(t)\rv_{L^2}\to0$ as $t\to\infty$ for each fixed $R$, the gradient term becomes arbitrarily small as $t$ gets large. Combining these estimates carefully yields $\lv w(t)\rv_{L^p}\to0$ as $t\to\infty$. 

The Gagliardo--Nirenberg inequality adapted to a strip is also required in Section \ref{s:decay}. We therefore formulate it in the slightly more general form needed there, involving $\Vert w \Vert_{L^p}^k $. 

\begin{proposition}[The Gagliardo--Nirenberg inequality for strips]\label{prop:G-N_strip} 

Suppose that $ w$ is a sufficiently smooth function defined in $\Omega := \OM$, subject to homogeneous Neumann boundary conditions on $y=\pm L$. Then, for any $R>0$, 
\[
\Vert w \Vert_{L^p(\Omega)}^k \leq  C\Big((2R+1)\Vert \nabla  w \Vert_{\LLO}^{k\theta}\lv  w\rv_{\LLO}^{k(1-\theta)} + \frac{(2R+1)^{k(1-\theta)/2}}{R^{k\theta/2}} \Vert  w \Vert_{\LLO}^k\Big),
\]
where $k >1$,  $0 < \theta < \min(2/k,1)$, and $p = 2/(1-\theta)$, and $C>0$ is some constant. 
\end{proposition}

\begin{proof}
We begin by extending $w$ to $\mathbb{R}^2$ as
\[
w_R(x,y) = \begin{cases}w(x,y),  &y \in (-L,L), \\
w\big(x,\frac{2L-y}{R}\big) g\big(\frac{y-L}{R}\big), &y \in (L,(2R+1)L), \\
w\big(x,-\frac{2L+y}{R}\big) g\big(-\frac{y+L}{R}\big), \quad &y \in (-L,-(2R+1)L), \\
0, &|y|\in((2R+1)L,\infty), \end{cases}
\]
where $g\in C^1(0,2L)$ is chosen so that $g(0)=1$, $g'(0)=0$ while $g(2L)=0$, $g'(2L)=0$, and $g$ is monotonically decreasing. For concreteness, we may choose $g(y) = \cos^2(\frac{\pi y}{4L})$. It is then straightforward to see that
\begin{equation}\label{eq:w_R}
\begin{aligned}
 &\Vert  w \Vert_{\LLO}^2 \leq \Vert w_R \Vert_{\LL}^2  \leq (2R+1) \Vert  w \Vert_{\LLO}^2,\\
 &\Vert \nabla  w \Vert_{\LLO}^2 \leq \Vert \nabla w_R \Vert_{\LL}^2 \leq (2R+1) \Vert \nabla  w\Vert^2_{\LLO} + \frac{2}{R} \max_{y\in(0,2L)}|g'(y)|^2 \Vert w \Vert_{\LLO}^2.
\end{aligned}
\end{equation}
When $g(y) = \cos^2(\frac{\pi y}{4L})$, $\max_{y\in(0,L)}|g'(y)|=\frac{\pi}{4L}$.

We now apply the Gagliardo--Nirenberg inequality on $\R^2$ to $w_R$, choosing $j = 0$, $m = 1$, $r = 2$, $q = 2$, and for $\theta$ and $p$ as given, and we utilize \eqref{eq:w_R} to obtain
\begin{align*}
\Vert & w_R \Vert_{L^p(\R^2)}^k \\
& \leq C \Vert \nabla w_R \Vert_{\LL}^{k\theta} \Vert w_R\Vert_{\LL}^{k(1-\theta)} \\
& \leq  C \Big((2R+1) \Vert \nabla  w\Vert^2_{\LLO} + \frac{2}{R} \Big(\frac{4}{\pi L}\Big)^2 \Vert w \Vert_{\LLO}^2 \Big)^{k\theta/2} (2R+1)^{k(1-\theta)/2} \Vert  w \Vert_{\LLO}^{k(1-\theta)}\\
&\leq C\Big((2R+1)\Vert \nabla  w \Vert_{\LLO}^{k\theta}\lv  w\rv_{\LLO}^{k(1-\theta)} + \frac{32^{k\theta/2}(2R+1)^{k(1-\theta)/2}}{(\pi L)^{k\theta}R^{k\theta/2}} \Vert  w \Vert_{\LLO}^k\Big)
\end{align*}
for some constant $C>0$. Since $\Vert w \Vert_{L^p(\Omega)}^k \leq \Vert w_R \Vert_{L^p(\R^2)}^k $, the proof is complete. 
\end{proof}

Therefore, although the standard Gagliardo--Nirenberg inequality on the whole plane is not directly applicable to a strip with homogeneous Neumann boundary conditions, it admits a suitable adaptation. Extending a function from the strip to $\mathbb{R}^2$ incurs an additional term proportional to $\lv w\rv_{L^2}$. For a sufficiently large $R$ and  an appropriate choice of $\theta$, the coefficient of this lower-order term can be made arbitrarily small, although not zero without making the coefficient of the $\|\nabla w\|_{L^2}$ term unbounded. 

Modified Gagliardo--Nirenberg inequalities are available on bounded domains with a lower-order penalty term $C\lv w\rv_{L^k}$ for arbitrary $k>0$. This arbitrariness relies on properties of $L^p$ spaces on bounded domains. The domains considered in Proposition \ref{prop:G-N_strip}, however, are unbounded and we do not claim that the exponent of the lower-order term is arbitrary here. 

We now turn to the proof of Theorem \ref{t:p>2}.

\begin{proof}[Proof of Theorem \ref{t:p>2}]
To simplify notation, throughout the proof, we write $\iint$ instead of $\int_{-\infty}^\infty\int_{-L}^L(\cdots)~dydx$.

\paragraph{Step 1:} We begin by calculating
\begin{align}
\half&\frac{d}{dt}(\lv w (t)\rv_{L^2} ^ {2} + \lv v(t) \rv_{L^2}^{2}) \notag\\
&= \iint ww_{t} + vv_{t} \notag\\ 
&\,\,\begin{aligned}= &\iint   ( \Delta w ) w + \nu w _ {xxx} w -( \phi w w_{x} +  \phi_{x}w ^ { 2 } + w^{2}w_{x} - \bm{\dot}{x_0}\phi_{x}w + v_{y}w) \\
& +\iint ( \Delta v ) v -  w_y v  \end{aligned} \notag\\ 
&=- \iint  (|\nabla w| ^ { 2 } +|\nabla v| ^ { 2 }) - \frac { 1 } { 2 } \iint \phi_{x} | w |^{2} -\gamma\Big(\iint\phi_{x}w\Big)^{2} \notag \\
&\,\,\begin{aligned}=& -\iint  \varepsilon(|\nabla w| ^ { 2 }+ |\nabla v| ^ { 2 })\\ 
&- \iint (1-\varepsilon) (|\nabla w|^2+ |\nabla v| ^ { 2 })- \frac { 1 } { 2 } \iint \phi_{x} | w |^{2} -\gamma\Big(\iint\phi_{x}w\Big)^{2} \end{aligned}\notag \\
& =: -\iint  \varepsilon(|\nabla v| ^ { 2 } +|\nabla w| ^ { 2 }) - \langle (w,v) , \mathcal{H}(w,v)\rangle,\label{e:step1}
\end{align}
where $\epsilon>0$ is sufficiently small. Throughout we adopt the shorthand notation $\bm{\dot}{x_0}=\frac{dx_0}{dt}$ where convenient. Here $\mathcal{H}:L^2(\mathbb{R}\times[-L,L]) \to L^2(\mathbb{R}\times[-L,L])$ is defined as 
\[ 
\mathcal{H}(w,v) = \Big(-(1-\varepsilon)\Delta w + \frac{1}{2}\phi_x(x)w + \gamma \mathcal{P}_{\phi_x}w, -(1-\varepsilon)\Delta v\Big), 
\] 
where $\mathcal{P}_{\phi_x}: L^2(\mathbb{R}\times[-L,L]) \to L^2(\mathbb{R}\times[-L,L])$ is the rank-one projection operator onto the span of $\phi_x$, defined as 
\[
\mathcal{P}_{\phi_x} f = \Big(\iint \phi_x(x) f(x,y)~dy dx\Big) \phi_x.
\]

We proceed to the associated eigenvalue problem
\begin{equation}\label{e:2Dop} 
\Big(- ( 1 - \varepsilon) \Delta + \frac { 1 } { 2 } \phi _ { x }(x) + \gamma \mathcal{P}_{\phi_x}\Big)w=\lambda w. 
\end{equation}
We write $w = \sum_{j = 0}^\infty W(x)a_j\cos(\frac{\pi y j}{L})$, so that \eqref{e:2Dop} becomes
\begin{equation}\label{e:eig1D}
\left\{\begin{aligned}
&- ( 1 - \varepsilon) \frac{d^2W}{dx^2} + \frac { 1 } { 2 } \phi _ { x }(x)W +2L\gamma \Big(\int\phi_{x}(x)W~dx\Big)^{2} = \lambda W, && j=0,\\
&- ( 1 - \varepsilon)\frac{d^2W}{dx^2} + \frac { 1 } { 2 } \phi _ { x }(x)W + (1-\varepsilon)\frac{\pi^2j^2}{4L^2}W = \lambda W, && j \geq 1.
\end{aligned}\right.
\end{equation}
Recall from \cite{BARKER2025} that 
\[-(1-\varepsilon)\frac{d^2}{dx^2}+ \frac { 1 } { 2 } \phi _ { x } (x)+ \gamma \mathcal{P}^{(1)}_{\phi_x},
\]
where $\mathcal{P}^{(1)}_{\phi_x}$ is the $L^2(\R)$ projection operator onto $\phi_x$, has no negative eigenvalues if $\gamma>0$ is sufficiently large, while $-(1-\varepsilon)\frac{d^2}{dx^2}+\frac12\phi_x(x)$ has exactly one negative eigenvalue by \eqref{e:stability}, denoted by $\lambda_{\min}$. 
Consequently, the first equation in \eqref{e:eig1D} has no negative eigenvalues for the same sufficiently large $\Gamma := 2L\gamma$. Fix such a $\gamma_*$ and set $2L\gamma=\gamma_*+\eps$.
The second equation in \eqref{e:eig1D} for $j=1$ then has no negative eigenvalues as long as $\frac{\pi^2}{4L^2}>-\lambda_{\min}$. Let $0 < L < L_{\max}$, where $L_{\max} = \frac{\pi}{2\sqrt{-\lambda_\textrm{min}}}$.
Furthermore, the second equation in \eqref{e:eig1D} for $j\geq2$ has no negative eigenvalues whenever the cases $j=0$ and $j=1$ have none. This holds as long as $L<L_{\max}$. To recapitulate, \eqref{e:2Dop} has no negative eigenvalues and, hence, $\langle (w,v) , \mathcal{H}(w,v)\rangle\geq0$, for sufficiently large $\gamma$ for $0<L<L_{\max}$.

Returning to the energy estimate, \eqref{e:step1} becomes
\begin{equation}\label{e:L2bounded} 
\frac{d}{dt}(\lv w \rv ^ { 2 }_{L^2} + \lv v \rv ^ { 2 }_{L^2}) \leq -\varepsilon \iint (|\nabla w| ^ { 2 } + |\nabla v| ^ { 2 }).
\end{equation}
Additionally, recalling $\gamma = \gamma_*+ \eps$, \eqref{e:step1} implies 
\begin{equation}\label{e:dxi0bounded} 
\frac{d}{dt}(\lv w \rv ^ { 2 }_{L^2} + \lv v \rv ^ { 2 }_{L^2}) \leq \varepsilon \Big(\iint\phi_{x}(x)w\Big)^{2} = \varepsilon \big|\bm{\dot}{x_0}\big|^2.\end{equation}
It follows from \eqref{e:L2bounded} that $\lv w(t)\rv_{L^2}^2 + \lv v(t)\rv_{L^2}^2$ is non-increasing and, hence, $\lv w(t) \rv_{L^2}$ and $\lv v(t) \rv_{L^2}$ remain bounded for $t\in[0,\infty)$. Furthermore, 
\begin{equation}\label{e:gradinL2}
\begin{aligned}
&\int_0^t ( \lv \nabla w(s) \rv_{L^2}^2 + \lv \nabla v(s) \rv_{L^2}^2)~ds \leq \frac{1}{2\eps}(\lv w_0\rv_{L^2}^2 + \lv v_0\rv_{L^2}^2-\lv w(t)\rv_{L^2}^2 -\lv v(t)\rv_{L^2}^2), \\
&\int_0^t \left| \frac{dx_0}{dt}(s)\right|^2~ds \leq \frac{1}{2\eps}(\lv w_0\rv_{L^2}^2 + \lv v_0\rv_{L^2}^2-\lv w( t)\rv_{L^2}^2 -\lv v(t)\rv_{L^2}^2).
\end{aligned}
\end{equation}
The right hand side of the first inequality is bounded independently of $t$, and therefore 
\[
\lv \nabla w \rv_{L^2}^2 + \lv \nabla v \rv_{L^2}^2 \in L^1([0,\infty);L^2(\OM)).
\]
Similarly, 
\[
\left| \frac{dx_0}{dt}\right| \in L^2([0,\infty);L^2(\OM)).
\]
The latter is not needed at this point but will be used in Section \ref{s:decay}. 

The maximal stability width $2L_{\max}$ depends implicitly on $\nu$ through the unique negative eigenvalue of $-(1-\varepsilon)\frac{d^2}{dx^2}+\frac12\phi_x(x)$. To indicate its magnitude, we numerically compute $\phi_x$ using a shooting method with $1.6\times10^6$ grid points and then numerically compute the eigenvalue $\lambda_{\min}$ of $-\frac{d^2}{dx^2}+\frac12\phi_x(x)$. Figure \ref{fig:Lmax_num} displays the resulting values of $L_{\max}$ for different values of $\nu$. 

\begin{figure}
    \centering
    \includegraphics[width=0.6\linewidth]{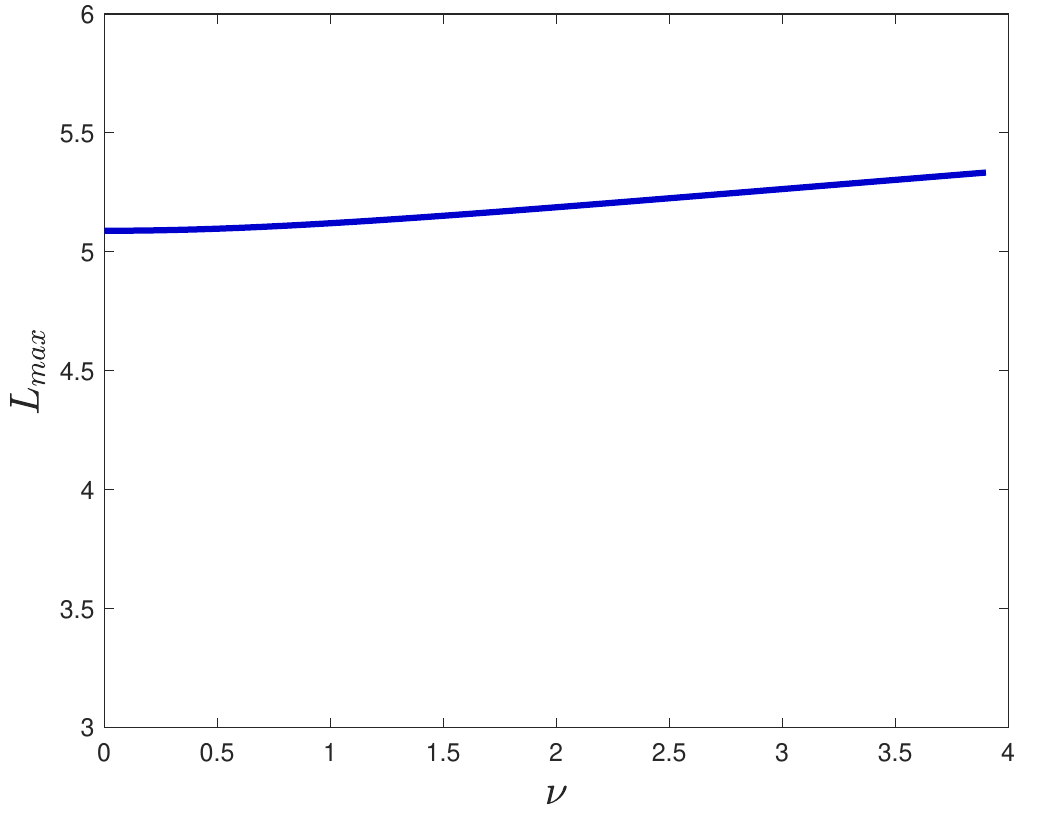}
    \caption{Numerically computed $L_{\textrm{max}}$ versus $\nu$.} 
    \label{fig:Lmax_num}
\end{figure}

\paragraph{Step 2:} We calculate 
\begin{align*} 
\half \frac{d}{dt}&(\|\nabla w\|_{L^2}^2+\|\nabla v\|_{L^2}^2) \\ 
= &\iint (w_{xx}w_t + w_{yy}w_t + v_{xx}v_t + v_{yy}v_t)~dxdy\\
= &-\iint w_{xx}(\phi w_x + w\phi_x + ww_x -\bm{\dot}{x_0} \phi_x + v_y - w_{xx} - w_{yy} {-\nu w_{xxx}})~dxdy \\
& - \iint w_{yy}(\phi w_x + \phi_xw + ww_x -\bm{\dot}{x_0} \phi_x + v_y - w_{xx} - w_{yy} {-\nu w_{xxx}})~dxdy \\
& + \iint (v_{xx} + v_{yy})^2-(v_{yy} w_y + v_{xx}w_y)~dxdy \\
   =& \gamma \iint \phi w_x  \iint \phi_{xx} w_x  -\iint \phi_{xx}ww_x - \frac{3}{2}\iint \phi_x w_x^2  - \frac12 \iint \phi_x w_y^2 \\
&+ \iint ww_xw_{xx} + \iint ww_xw_{yy}- \iint (w_{xx} + w_{yy})^2- \iint (v_{xx} + v_{yy})^2. 
\end{align*}
A straightforward calculation then reveals:
\begin{align*}
&\,\begin{aligned}\left|\gamma \iint \phi w_x  \iint \phi_{xx} w_x \right| &= \gamma \left| \iint \phi_x w  \iint \phi_{xx} w_x \right| \\ 
&\leq \gamma \lv \phi_x \rv_{L^\infty} \lv w \rv_{L^2} \lv \phi_{xx} \rv_{L^\infty} \lv w_x \rv_{L^2} \leq C_1\lv \nabla w \rv_{L^2}, \end{aligned}\\
&\left|\iint \phi_{xx}ww_x\right| \leq   \lv \phi_{xx} \rv_{L^\infty} \lv w \rv_{L^2}\lv w_x \rv_{L^2} \leq C_1\lv \nabla w \rv_{L^2}, \\
&\left|\frac{3}{2}\iint \phi_x w_x^2\right| + \left|\frac12 \iint \phi_x w_y^2\right| \leq \frac{3}{2}\lv \phi_x \rv_{L^\infty} \lv w_x \rv_{L^2}^2 + \frac12 \lv \phi_x \rv_{L^\infty} \lv w_y \rv_{L^2}^2\leq C_2 \lv \nabla w \rv_{L^2}^2, \\
&\,\begin{aligned}\left|\iint ww_xw_{xx}\right| +  \left|\iint ww_xw_{yy}\right|  &= \frac12  \left|\iint w_x^3\right| + \frac12  \left| \iint w_xw_y^2\right| \\
&\leq \lv \nabla w \rv_{L^3}^3 \leq C_4\lv \Delta w\rv_{L^2} \lv \nabla w \rv_{L^2}^2 + C_3\lv \nabla w \rv_{L^2}^3,\end{aligned}
\end{align*}
where $C_j$, $j=1,2,3,4$, are positive constants. For the last inequality,
we apply Proposition \ref{prop:G-N_strip} to $\nabla w $ with $p = 3$,  $\theta =2/3$, and $R=1$. Putting these together, 
\begin{align*}
\frac12\frac{d}{dt}&(\lv \nabla w \rv_{L^2}^2 + \lv \nabla v \rv_{L^2}^2) \\
&\leq C_1 \lv \nabla w \rv_{L^2} + C_2 \lv \nabla w \rv_{L^2}^2 + C_3\lv \nabla w \rv_{L^2}^3 +  C_4\lv \Delta w\rv_{L^2} \lv \nabla w \rv_{L^2}^2\\ 
& \quad -\lv \Delta w \rv_{L^2}^2  -\lv \Delta v \rv_{L^2}^2 \\
&\leq  C_1 \lv \nabla w \rv_{L^2} + C_2 \lv \nabla w \rv_{L^2}^2 + C_3\lv \nabla w \rv_{L^2}^3  + \frac{C_4}{4}\lv \nabla w \rv_{L^2}^4\\ 
&\quad -\Big(\lv \Delta w \rv_{L^2} -  \frac{C_4}{2} \lv \nabla w \rv_{L^2}^2\Big)^2\\
 &\leq  C_1 \lv \nabla w \rv_{L^2} + C_2 \lv \nabla w \rv_{L^2}^2 + C_3\lv \nabla w \rv_{L^2}^3 +C_4\lv \nabla w \rv_{L^2}^4. 
\end{align*}

Let 
\[
\zeta(t) = \lv \nabla w(t) \rv_{L^2}^2 + \lv \nabla v(t) \rv_{L^2}^2,
\]
so that the above becomes
\begin{equation}\label{e:zeta}
\frac12 \frac{d\zeta}{dt} \leq C_1 \zeta^\frac12 + C_2 \zeta + C_3\zeta^{\frac{3}{2}} + {C_4}\zeta^2.
\end{equation}

Since $\zeta \in L^1([0,\infty))$, for every $\delta > 0$ there exists some $T > 0$ sufficiently large such that 
\[
\zeta(T) < \delta \quad\text{and}\quad \int_T^\infty \zeta(t)dt < \delta.
\]
Let $T' > T$, and 
\[
M:= \max_{t \in (T,T')}\zeta(t).
\]
Multiplying \eqref{e:zeta} by $\zeta^{1/2}$ and integrating over $(T,T')$, we obtain
\begin{align*}
\frac{1}{3}(M^{3/2}-\delta^{3/2}) &\leq \frac12 \int_T^{T'}\zeta^{1/2} \left| \frac{d\zeta}{dt}\right|~dt \\
&\leq \int_T^{T'} (C_1 \zeta + C_2 \zeta^{3/2} + C_3\zeta^2 + {C_4}\zeta^{5/2})~dt\\
& \leq (C_1 + C_2M^{1/2} + C_3 M + {C_4} M^{3/2})\int_T^{T'}\zeta~dt\\
& \leq (C_1 + C_2M^{1/2} + C_3 M + {C_4} M^{3/2})\int_T^{\infty}\zeta~dt\\
&\leq C_1\delta + C_2\delta M^{1/2} + C_3\delta M + {C_4}\delta M^{3/2}.
\end{align*}
Consequently, 

\[
M^{3/2} \leq \frac{\delta^{3/2} + 3 C_1\delta + C_2\delta M^{1/2} + C_3\delta M}{1 -  {C_4}\delta} \leq C\delta^{3/2}
\]
for $\delta>0$ sufficiently small, independent of $T'$, implying 
\[
\zeta(t)= \lv \nabla w(t)\rv_{L^2}^2 + \lv \nabla v(t) \rv_{L^2}^2 \to 0 \quad\text{as $ t \to \infty$}.
\]

\paragraph{Step 3:}It remains to show that $w,v\to 0$ as $t\to\infty$ in appropriate $L^p$ norms.

An application of Proposition \ref{prop:G-N_strip}, with $\theta > 1/2$ and $p = 2/(1-\theta)( > 4)$, yields
\[
\Vert w \Vert_{L^p}^2 \leq  C\Big((2R+1)\Vert \nabla  w \Vert_{L^2}^{2\theta}\lv  w\rv_{L^2}^{2-2\theta} + \frac{(2R+1)^{1-\theta}}{R^{\theta}} \Vert  w \Vert_{L^2}^2\Big).
\]
Since $\Vert  w(t) \Vert_{L^2}^2$ is bounded for $t\in[0,\infty)$ and since $\theta > 1/2$, for every $\eps > 0$ there exists $R_0 > 0$ sufficiently large such that 
\[
C\frac{(2R_0+1)^{1-\theta}}{R_0^\theta} \Vert  w \Vert_{L^2}^2 < \frac{\eps}{2} .
\]

We fix such $R_0$. Since $\Vert \nabla  w(t) \Vert_{L^2} \to 0$ as $t\to \infty$, there exists $T>0$ sufficiently large such that 
\[C(2R_0+1)\Vert \nabla  w \Vert_{L^2}^{2\theta}\lv  w\rv_{L^2}^{2-2\theta} < \frac{\eps}{2} \quad \text{for $t > T$} .
\] 
Putting these together, for every $\eps > 0$ there exists $T > 0$ sufficiently large such that $\Vert w(t) \Vert_{L^p}^2 < \eps$ for $t > T$. Therefore,
\[
\Vert w(t) \Vert_{L^p} \to 0\quad \text{as $t \to \infty$}.
\]
This holds for all $\theta \in(1/2, 1)$ or for all $p \in (4, \infty)$. 
We then interpolate between $L^2$ and $L^p$ using the boundedness of $\Vert  w \Vert_{L^2}$, to obtain 
\[
\Vert w(t) \Vert_{L^p} \to 0\quad \text{as $t \to \infty$} \quad\text{for $ p \in(2, \infty)$}. 
\]

On the other hand, since $v$ satisfies homogeneous Dirichlet boundary conditions in $y$, it follows from the Poincar{\'e} inequality that
\[
\lv v \rv_{L^2} \leq C \lv \nabla v \rv_{L^2}.
\]
Moreover, the standard Gagliardo--Nirenberg inequality yields 
\[
\lv v \rv_{L^p} \leq C\lv \nabla v\rv_{L^2}^{1-2/p} \lv v \rv_{L^2}^{2/p} \quad \text{for $p \in (2, \infty)$}.
\]
Since $\lv v(t) \rv_{L^2}$ is bounded for $t\in[0,\infty)$ and since $\lv \nabla v(t) \rv_{L^2} \to 0$ as $t \to \infty$, it follows that $\lv v(t) \rv_{L^p} \to 0$ as $t\to\infty$ for $p \in (2, \infty)$.

This completes the proof.
\end{proof}

\section{Stability and algebraic decay in \texorpdfstring{$L^p$}{} for \texorpdfstring{$p \in (1,2]$}{} }\label{s:decay}

We build on Section \ref{s:stab} to establish stability under the $L^p$ norm for $p \in (1,2]$, along with algebraic decay rates. The argument follows the framework of \cite{stanislavova2025asymptotic}, developed for \eqref{e:KdVB}, with necessary modifications for a two-dimensional strip.

\begin{hyp}[Additional localization of initial data]\label{hyp:2}
Let $(u_0,v_0)$ denote the initial condition for \eqref{e:Rajopadhye} and \eqref{e:BC} (see \eqref{e:IC}), and $\phi$ satisfy \eqref{e:KdVB0}. We assume that 
\begin{align*}
    &\iint_{\R \times [-L,L]} |u_0(x,y) - \phi(x)|^2(|x|+1)~dxdy < \infty,\\
    &\iint_{\R \times [-L,L]} |v_0(x,y)|^2(|x|+1)~dxdy < \infty. \\
\end{align*}
\end{hyp}

\begin{theorem}[Stability and algebraic decay]\label{t:stab_Lp_1<p<2}
Let $( u, v)$ denote a solution of \eqref{e:Rajopadhye}, \eqref{e:BC}, and \eqref{e:IC}, satisfying Hypotheses \ref{hyp:1} and \ref{hyp:2}, and $\phi$ a solution of \eqref{e:KdVB0}. Suppose that \eqref{e:xi_0} and \eqref{e:stability} hold, where $\gamma>0$ is chosen in the proof of Theorem \ref{t:p>2}, and $\lambda_{\min}$ is the unique negative eigenvalue of $-\frac{d^2}{dx^2} + \frac{1}{2}\phi_x(x)$. 

Let $L < L_{\max} := \frac{\pi}{2\sqrt{\lambda_{\min}}}$. Then 
\[
\lv w(t)\rv_{L^2}+\lv v(t)\rv_{L^2} \leq Ct^{-1/2} \quad \text{as $t\to\infty$}
\]
for some constant $C>0$. Moreover, for $\delta > 0$ sufficiently small,
\[
\lv w(t) \rv_{L^p}+\lv v(t) \rv_{L^p} \leq Ct^{-(1-1/p)+\delta} \quad \text{as $t\to\infty$}\quad \text{for $p \in (1, 2)$}
\]
for some constant $C>0$, depending on $\delta$.
\end{theorem}

\begin{remark*}\rm
The algebraic decay rates in Theorem \ref{t:stab_Lp_1<p<2} do not contradict Pego's observation \cite{Pego} that no uniform decay estimate of the form 
\[
\lv w(t)\rv_1\leq C(t)\lv w(0)\rv_2, \qquad \text{$C(t)\to0$ as $t\to\infty$},
\]
can hold for \eqref{e:KdVB} if both $\lv \cdot \rv_1$ and $\lv \cdot \rv_2$ are invariant under translations. In our proof below, the additive constant in \eqref{e:bound_withV} contains $\iint(w_0^2 + v_0^2)(2|x| + 1)$, which is not translation invariant. Accordingly, $C$ in Theorem \ref{t:stab_Lp_1<p<2} depends on weighted norms of $u_0$ and $v_0$. 
\end{remark*}

We outline the proof of Theorem \ref{t:stab_Lp_1<p<2}. The main task is to show that $\lv w(t)\rv_{L^2}^2$ is integrable. Since $\lv w(t)\rv_{L^2}^2$ is decreasing by Section \ref{s:stab}, this would imply $\lv w(t)\rv_{L^2}\leq Ct^{-1/2}$. On the other hand, since $\int_0^\infty\lv \nabla v(t) \rv_{L^2}^2dt < \infty$ by Section \ref{s:stab}, it follows from the Poincar{\'e} inequality that $\lv v(t)\rv_{L^2}^2$ is integrable.

To prove the integrability of $\lv w(t)\rv_{L^2}^2$, we develop a weighted energy estimate for 
\[
\iint(w^2+v^2)|x|~dxdy.
\]
An important contribution comes from $\iint \phi w_xw|x|$, which becomes $\iint-\sgn(x)\phi(x)w^2$ after integration by parts. See \eqref{e:energy|x|}. Observing $-\sgn(x)\phi(x)\to1$ as $|x|\to\infty$, we bound this by $\iint w^2$ up to an error involving $\|w\|_{L^5}$. We then bound the remaining terms mainly by two time-integrable quantities established in Section \ref{s:stab}. First, $\lv\nabla w(t)\rv_{L^2}^2+\lv\nabla v(t)\rv_{L^2}^2$ are integrable.  Second, $\int_0^t|\bm{\dot}{x_0}(s)|^2~ds<\infty$ by \eqref{e:gradinL2}. The Gagliardo--Nirenberg inequality adapted to a strip controls the resulting $L^5$ terms.

Integration by parts also produces a boundary term at $x=0$, for which a pointwise $L^\infty$ control needed for a direct estimate is unavailable. Instead, we repeat the weighted energy estimate with $|x-a|$ and average the result over $a\in(-1,1)$, thereby converting pointwise norms into integral ones.

We then integrate $\|w(t)\|_{L^2}^2$ and $\frac{d}{dt}\iint (w^2 + v^2)|x| $, which are bounded by quantities that are integrable in $t$, to show that $\int_0^\infty\lv w(t)\rv_{L^2}^2dt$ and $\iint(w^2+v^2)|x|$ remain bounded. The former bound yields the decay rate in $L^2$, while the latter provides the spatial localization needed for interpolation and, hence, the decay rate in $L^p$ for $p\in(1,2)$. 

Importantly, this closes the range of exponents left open by Section \ref{s:stab}.

\newpage

\begin{proof}
To streamline notation, throughout the proof, we use $\iint$ for $\int_{-\infty}^\infty\int_{-L}^L(\cdots)~dydx$. 

We begin by calculating
\begin{align*}
\frac{1}{2}\frac{d}{dt}& \iint (w^2 + v^2)|x|  \\
=&  -\iint ( \phi w w_{x} +  \phi_{x}w ^ { 2 } + w^{2}w_{x} - \bm{\dot}{x_0}\phi_{x}w + v_{y}w)|x| +  \iint (( \Delta w ) w + \nu w _ {xxx} w)|x| \\ 
& + \iint ((\Delta v)v - vw_y)|x|   \\
= &\iint (\Delta w) w|x| + \iint (\Delta v) v|x|+ \bm{\dot}{x_0}\iint \phi_xw|x| - \iint \phi_xw^2|x| \\
&- \iint \phi w_x w|x| - \iint w_x w^2|x| + \nu \iint w_{xxx}w|x|.
\end{align*}
After integrating $\displaystyle \iint \phi w_xw|x|$ by parts in $x$, 
\begin{equation}\label{e:energy|x|}
\begin{split}
        \frac{1}{2} \frac{d}{dt} \iint (w^2 + v^2)|x| -& \half \iint \phi w^2 \sgn(x)\\
        =& \iint ((\Delta w) w+ (\Delta v) v)|x|+ \bm{\dot}{x_0}\iint \phi_xw|x|  \\
& - \half\iint \phi_xw^2|x| - \iint w_x w^2|x| + \nu \iint w_{xxx}w|x|.
    \end{split}
\end{equation}
Since $-\phi(x)\sgn(x) > 0$ \cite{BS1985} and since $-\phi(x)\sgn(x) \to 1$ as $|x| \to \infty$, by \eqref{e:KdVB0}, there must exist $A > 0$ such that $-\phi(x)\sgn(x) > \frac{1}{2}$ for $|x| > A$. Accordingly,
\begin{align*}
- \frac{1}{2} \iint \phi w^2 \sgn(x) > -\frac{1}{2}\iint_{|x|>A} \phi w^2 \sgn(x) >& \frac{1}{4}\iint_{|x|>A} w^2 \\
=&\frac{1}{4}\iint w^2 - \frac{1}{4}\iint_{|x| < A}w^2 \\  \geq & \frac{1}{4}\iint w^2 - \frac{1}{4}\lv 1_{[-A,A]}\rv_{L^{5/3}}\lv w^2 \rv_{L^{5/2}} \\
\geq & \frac{1}{4}\iint w^2  - C \lv w \rv_{L^5}^2
\end{align*}
for some constant $C>0$. In other words, we replace $\displaystyle - \frac{1}{2} \iint \phi w^2 \sgn(x)$ on the left hand side of \eqref{e:energy|x|} by $\displaystyle \frac{1}{4}\iint w^2 $ with the error term proportional to $ \lv w \rv_{L^5}^2$, which we will address later. 

We then handle the remaining terms on the right hand side of \eqref{e:energy|x|}, using where necessary integration by parts in $x$ and H{\"o}lder's inequality. We collect our result: 
\begin{align*}
&\iint (\Delta w)w|x| = -\iint (w_x^2+ w_y^2)|x| -\iint w_x w \sgn(x) \leq \lv \nabla w \rv_{L^2} \lv w \rv_{L^2}, \\
&\iint (\Delta v)v|x| = -\iint (v_x^2+ v_y^2)|x| -\iint v_x v \sgn(x) \leq \lv \nabla v \rv_{L^2} \lv v\rv_{L^2}, \\
&\bm{\dot}{x_0}\iint \phi_xw|x| \leq |\bm{\dot}{x_0}|\lv \phi_x |x|\rv_{L^2} \lv w\rv_{L^2} \leq C|\bm{\dot}{x_0}|\lv w\rv_{L^2} ,\\
&-\frac{1}{2} \iint \phi_x w^2 |x| \leq \lv \phi_x |x|\rv_{L^{5/3}} \lv w^2\rv_{L^{5/2}}  \leq C\lv w\rv_{L^5}^2,\\
&\iint w_xw^2|x| = \frac{1}{3}\iint w^3\sgn(x) \leq \lv w\rv_{L^3}^3 \leq  \lv w \rv^{2/3}_{L^2} \lv w \rv^{7/3}_{L^5} \leq C\lv w \rv^{7/3}_{L^5} , 
\intertext{and}
&\begin{aligned}\iint w_{xxx}w|x| =& \frac{1}{2}\iint (w_x)^2\sgn(x) - \iint w_{xx}w\sgn(x) \\
\leq& C\int_{-L}^L|w(0,y,t)||w_x(0,y,t)|~dy + C\lv \nabla w \rv_{L^2}^2 ,\end{aligned} 
\end{align*}
where $C>0$ is a constant. Putting these together, 
\begin{equation}\label{e:est_withL5}
\begin{split}
\frac{d}{dt}\iint (w^2 +& v^2)|x| + \frac{1}{4}\iint w^2 \\
\leq &C_1 \bigg( |\bm{\dot}{x_0}|\lv w\rv_{L^2} + \int_{-L}^L|w(0,y,t)||w_x(0,y,t)|~dy \\
&\qquad + \lv \nabla w \rv_{L^2}^2 + \lv \nabla w \rv_{L^2}\lv w \rv_{L^2}  + \lv \nabla v \rv_{L^2} \lv v\rv_{L^2} + \lv w\rv_{L^5}^2 + \lv w\rv_{L^5}^{7/3}\bigg)
\end{split}
\end{equation}
for some constant $ C_1>0$.

Proposition \ref{prop:G-N_strip} applies to $\lv w\rv_{L^5}$ to yield: 
\begin{equation}\label{e:GN_fromL5toL2}
\begin{split}
&\Vert w \Vert_{L^5}^2 \leq C(2R+1)\Vert \nabla  w \Vert_{L^2}^{6/5}\lv w\rv_{L^2}^{4/5} + C\frac{(2R+1)^{2/5}}{R^{3/5}} \Vert  w \Vert_{L^2}^2,
\\
&\Vert w \Vert_{L^5}^{7/3} \leq C(2R+1)\Vert \nabla  w \Vert_{L^2}^{7/5}\lv w\rv_{L^2}^{14/15} + C\frac{(2R+1)^{7/15}}{R^{7/10}} \Vert  w \Vert_{L^2}^{7/3},
\end{split}
\end{equation}
for some constant $C>0$ for any $R > 0$. We fix $R_0 > 0$  sufficiently large such that 
\[ 
C_1\left(C\frac{(2R_0+1)^{7/15}}{R_0^{7/10}}\lv w_0 \rv_{L^2}^{1/3} + C\frac{(2R_0+1)^{2/5}}{R_0^{3/5}}\right) < \frac{1}{8}.
\] 
Substituting \eqref{e:GN_fromL5toL2} into \eqref{e:est_withL5} and  subtracting $\lv w \rv_{L^2}^2$, we then arrive at
\begin{align*}
\frac{d}{dt}&\iint (w^2 + v^2)|x|+ \frac{1}{8}\iint w^2 \\
&\begin{aligned}\leq C\bigg( &|\bm{\dot}{x_0}|\lv w\rv_{L^2} + \int_{-L}^L|w(0,y,t)||w_x(0,y,t)|~dy + \lv \nabla w \rv_{L^2}^2\\
& + \lv \nabla w \rv_{L^2}\lv w \rv_{L^2} + \lv \nabla v \rv_{L^2} \lv v\rv_{L^2} +  \Vert \nabla  w \Vert_{L^2}^{6/5}\lv w\rv_{L^2}^{4/5} +  \Vert \nabla  w \Vert_{L^2}^{7/5}\lv w\rv_{L^2}^{14/15}\bigg)\end{aligned} \\
&\begin{aligned}\leq C\bigg(& |\bm{\dot}{x_0}|\lv w\rv_{L^2} + \int_{-L}^L|w(0,y,t)||w_x(0,y,t)|~dy + \lv \nabla w \rv_{L^2}^2\\
&+ \lv \nabla w \rv_{L^2}\lv w \rv_{L^2} + \lv \nabla v \rv_{L^2} \lv v\rv_{L^2} +  \Vert \nabla  w \Vert_{L^2}^{6/5}\lv w\rv_{L^2}^{4/5} +  \Vert \nabla  w \Vert_{L^2}^{7/5}\lv w\rv_{L^2}^{3/5}\bigg)\end{aligned} 
\end{align*}
for some constant $C>0$.
We emphasize that each term on the right hand side contains a factor that is integrable in $t$, except for $\displaystyle \int_{-L}^L|w(0,y,t)||w_x(0,y,t)|~dy$, which comes from a boundary term at $x=0$ upon integration by parts. We will devise an averaging argument to turn this pointwise norm into an integral.

\paragraph{Averaging.} 

A direct estimate of  $\displaystyle \int_{-L}^L|w(0,y,t)||w_x(0,y,t)|~dy$ would require an $L^\infty$ control of $w$ or $\nabla w$, which is unfortunately unavailable. On the other hand, the point $x=0$ is not distinguished. Indeed, repeating the above weighted estimate for $\iint (w^2 + v^2)|x-a| $ for $a \in (-1,1)$, we arrive at
\begin{equation*}\label{e:avg1}
\begin{split}
\frac{d}{dt}&\iint (w^2 + v^2)|x-a| + \frac{1}{8}\iint w^2 \\
&\begin{aligned}\leq C\big(& |\bm{\dot}{x_0}|\lv w\rv_{L^2} + |w(a,y,t)||w_x(a,y,t)| + \lv \nabla w \rv_{L^2}^2 \\
&+ \lv \nabla w \rv_{L^2}\lv w \rv_{L^2} + \lv \nabla v \rv_{L^2} \lv v\rv_{L^2}+  \Vert \nabla  w \Vert_{L^2}^{6/5}\lv w\rv_{L^2}^{4/5} +  \Vert \nabla  w \Vert_{L^2}^{7/5}\lv w\rv_{L^2}^{3/5} \big)\end{aligned} 
\end{split}
\end{equation*}
for some constant $C>0$. Integrating this over $(0,t)$,
\begin{align*}
\iint & (w^2 + v^2)|x-a|-\iint (w_0^2 + v_0^2)|x-a| + \frac{1}{8}\int_0^t\iint w^2 \\
\leq & C\int_0^t \bigg(|\bm{\dot}{x_0}|\lv w(s)\rv_{L^2}+\int_{-L}^L|w(a,y,s)||w_x(a,y,s)|~dy \bigg)~ds \\
&+C\int_0^t\big(\lv \nabla w(s) \rv_{L^2}\lv w(s) \rv_{L^2}
+  \lv \nabla v(s) \rv_{L^2} \lv v(s)\rv_{L^2} \big)~ds\\
&+ C\int_0^t \big( \lv \nabla w(s) \rv_{L^2}^2  + \Vert \nabla  w(s) \Vert_{L^2}^{6/5}\lv w(s)\rv_{L^2}^{4/5} +  \Vert \nabla  w (s)\Vert_{L^2}^{7/5}\lv w(s)\rv_{L^2}^{3/5}\big)~ds . 
\end{align*}
Integrating this over $(-1,1)$ with respect to $a$, and using 
\[
\int_{-1}^1|x-a|~da = \begin{cases} 2|x|, &|x| > 1 \\x^2 + 1, &|x| \leq 1\end{cases}\  \in \ (2|x|, 2|x|+1),
\]

we ultimately arrive at
\begin{align*}
\iint &(w^2 + v^2)|x|-\iint (w_0^2 + v_0^2)(2|x|+1) + \frac{1}{4}\int_0^t\iint w^2 \\
\leq &C\bigg(\int_0^t |\bm{\dot}{x_0}|\lv w(s)\rv_{L^2}~ds+\int_{-1}^1\int_{-L}^L\int_0^t  |w(a,y,s)||w_x(a,y,s)|~dsdyda\bigg)\\
&+C\int_0^t\big(\lv \nabla w(s) \rv_{L^2}\lv w(s) \rv_{L^2} 
 +  \lv \nabla v(s) \rv_{L^2} \lv v(s)\rv_{L^2}\big)~ds\\
&+ C\int_0^t \big( \lv \nabla w(s) \rv_{L^2}^2  + \Vert \nabla  w(s) \Vert_{L^2}^{6/5}\lv w(s)\rv_{L^2}^{4/5} +  \Vert \nabla  w (s)\Vert_{L^2}^{7/5}\lv w(s)\rv_{L^2}^{3/5}\big)~ds . 
\end{align*}

We now bound the right hand side in terms of time-integrable quantities established earlier in Section \ref{s:stab} and, particularly, in terms of 
\[
\int_0^t \lv \nabla w(s)\rv_{L^2}^2~ds, \quad \int_0^t \lv \nabla v(s)\rv_{L^2}^2~ds,\quad\text{and}\quad \int_0^t|\bm{\dot}{x_0}(s)|^2~ds
\]
(see \eqref{e:gradinL2}), using H{\"o}lder's inequality. We collect our result: 
\begin{align*}
&\int_0^t \lv \nabla w(s)\rv_{L^2}^2~ds \leq  C ,\\
&\int_0^t|\bm{\dot}{x_0}|\lv w(s)\rv_{L^2}~ds \leq \Big(\int_0^t|\bm{\dot}{x_0}(s)|^2~ds\Big)^{1/2} \Big(\int_0^t \lv w(s)\rv_{L^2}^2~ds\Big)^{1/2} \leq C \lv w \rv_{L^2_{t,x,y}}   ,\\
&\,\begin{aligned}\int_0^t  \lv \nabla w(s) \rv_{L^2}&\lv w(s) \rv_{L^2}~ds \\ &\leq  \Big(\int_0^t \lv \nabla w(s)\rv_{L^2}^2~ds\Big)^{1/2} \Big(\int_0^t \lv w(s)\rv_{L^2}^2~ds\Big)^{1/2} \leq C \lv w \rv_{L^2_{t,x,y}},\end{aligned} \\
&\,\begin{aligned}\int_0^t  \lv \nabla v(s) \rv_{L^2}&\lv v(s) \rv_{L^2}~ds \\ &\leq  \Big(\int_0^t \lv \nabla v(s)\rv_{L^2}^2~ds\Big)^{1/2} \Big(\int_0^t \lv v(s)\rv_{L^2}^2~ds\Big)^{1/2} \leq C \lv v \rv_{L^2_{t,x,y}},\end{aligned}\\
&\,\begin{aligned}\int_0^t \Vert \nabla  w(s) \Vert_{L^2}^{6/5}&\lv w(s)\rv_{L^2}^{4/5}~ds  \\ & \leq  \Big(\int_0^t|\bm{\dot}{x_0}(s)|^2~ds\Big)^{3/5} \Big(\int_0^t \lv w(s)\rv_{L^2}^2~ds\Big)^{2/5} \leq C \lv w \rv_{L^2_{t,x,y}}^{4/5},\end{aligned} \\
&\,\begin{aligned}\int_0^t \Vert \nabla  w(s) \Vert_{L^2}^{7/5}&\lv w(s)\rv_{L^2}^{7/5}~ds  \\
& \leq  \Big(\int_0^t|\bm{\dot}{x_0}(s)|^2~ds\Big)^{7/10} \Big(\int_0^t \lv w(s)\rv_{L^2}^2~ds\Big)^{3/10} \leq C \lv w \rv_{L^2_{t,x,y}}^{3/5} ,\end{aligned} 
\intertext{and}
&\,\begin{aligned}\int_{-1}^1\int_{-L}^L\int_0^t  & |w(a,y,s)||w_x(a,y,s)|~dsdyda \\
&\leq   \Big(\int_0^t \lv \nabla w(s)\rv_{L^2}^2~ds\Big)^{1/2} \Big(\int_0^t \lv w(s)\rv_{L^2}^2~ds\Big)^{1/2} \leq C \lv w \rv_{L^2_{t,x,y}},\end{aligned}
\end{align*}
where $\displaystyle \lv w \rv_{L^2_{t,x,y}} = \Big(\int_0^t \lv w(s) \rv_{L^2}^2~ds\Big)^{1/2}$. 

Therefore we conclude that 
\begin{equation}
    \iint (w^2 + v^2)|x| + \frac{1}{4}\lv w \rv_{L^2_{t,x,y}}^2 \leq C\big(\lv w \rv_{L^2_{t,x,y}}^{3/5} + \lv w \rv_{L^2_{t,x,y}}^{4/5}  + \lv w \rv_{L^2_{t,x,y}}+ \lv v \rv_{L^2_{t,x,y}}\big)  + C.
\end{equation}
Furthermore, it follows from Young's inequality that:
\begin{align*}
&C\lv w \rv_{L^2_{t,x,y}} \leq \frac{1}{32}\lv w \rv_{L^2_{t,x,y}}^2 + 8C^2  \leq \frac{1}{32}\lv w \rv_{L^2_{t,x,y}}^2 + C_1,\\
&C\lv w \rv_{L^2_{t,x,y}}^{4/5}  \leq \frac{1}{80}\lv w \rv_{L^2_{t,x,y}}^2 + \frac{3}{5}(4C)^{5/3} \leq \frac{1}{32}\lv w \rv_{L^2_{t,x,y}}^2 + C_1,\\
&C \lv w \rv_{L^2_{t,x,y}}^{3/5} \leq \frac{3}{10\cdot 2^{10}}\lv w \rv_{L^2_{t,x,y}}^2 + \frac{7}{10}(8C)^{10/7} \leq \frac{1}{32} \lv w \rv_{L^2_{t,x,y}}^2+ C_1,
\end{align*}
where $C_1>0$ is some constant. We then subtract the powers of $\lv w \rv_{L^2_{t,x,y}}$ from the right hand side over to the left, which yields
\begin{equation}\label{e:bound_withV}
 \iint (w^2 + v^2)|x| + \frac{1}{8}\lv w \rv_{L^2_{t,x,y}}^2 \leq C\lv v \rv_{L^2_{t,x,y}}  + C.     
\end{equation}
Since $\lv \nabla v \rv_{L^2_{t,x,y}}$ is bounded, it follows from the Poincar{\'e} inequality that $\lv v \rv_{L^2_{t,x,y}}$ is bounded. Therefore, \eqref{e:bound_withV} ultimately becomes 
\begin{equation}\label{e:bound_noV}
 \iint (w^2 + v^2)|x| + \frac{1}{8}\lv w \rv_{L^2_{t,x,y}}^2 \leq C.  
\end{equation}

Since $\lv w(t) \rv_{L^2}^2 + \lv v(t)\rv_{L^2}^2$ is decreasing and integrable in $t$, by \eqref{e:L2bounded}, an integration yields
\[
t(\lv v \rv_{L^2}^2 + \lv w\rv_{L^2}^2) \leq \int_0^t(\lv v(s) \rv_{L^2}^2 + \lv w(s)\rv_{L^2}^2)~ds \leq \lv w \rv_{L^2_{t,x,y}}^2+\lv v \rv_{L^2_{t,x,y}}^2 \leq C.
\]
Consequently,
\begin{equation}\label{e:L2_decay}
\lv w(t)\rv_{L^2}+\lv v(t)\rv_{L^2} \leq Ct^{-1/2}.
\end{equation}
We therefore establish boundedness and decay in $L^2$. 

When $1 < p < 2$, a straightforward calculation yields: 
\begin{equation}
\begin{split}\label{e:Lp_unif_bound}
    \iint |w(x,y)|^p &= \iint_{|x|<1}|w(x,y)|^p + \iint_{|x|\geq1}|w(x,y)|^p \\
    &\leq C \lv w \rv_{L^2}^p + \Big(\iint_{|x|\geq1}w^2(x,y)|x| \Big)^{p/2}\Big(\iint_{|x|\geq1}|x|^{-p/(2-p)} \Big)^{1-(p/2)} \\
    &\leq C(p)\Big(\iint_{|x|\geq1}w^2(x,y)|x| \Big)^{p/2},
    \end{split}
\end{equation}
where the right hand side is bounded by \eqref{e:bound_noV} uniformly in time. The same argument applies to $v$.  We pause to remark that $C(p) \to \infty$ as $p \to 1$ and, hence, this is not valid for $p = 1$. 

Finally, we combine \eqref{e:Lp_unif_bound} and \eqref{e:L2_decay} and use interpolation inequalities for $L^p$ spaces to show that 
for every $\delta > 0$ there exists $C(\delta)$ such that 
\[
\lv w(t) \rv_{L^p}+\lv v(t) \rv_{L^p} \leq Ct^{-(1-1/p)+\delta} \quad \text{for $p\in (1,2)$}. 
\]
We therefore establish boundedness and decay in $L^p$ for $p\in (1,2]$. 

This completes the proof.
\end{proof}

\section{Concluding remarks}

We have established global asymptotic stability of one-dimensional KdVB fronts in a weakly two-dimensional setting for sufficiently narrow channels, encompassing both monotone and non-monotone profiles. The restriction on the channel width is a technical requirement that ensures the auxiliary operator $\mathcal{H}$ (see \eqref{e:step1}) is positive semi-definite. It is therefore interesting to ask whether this restriction is necessary, or whether instability sets in when $L>L_{\max}$. When implementing direct numerical simulation, we see no qualitative difference between the cases $L<L_{\max}$ and $L>L_{\max}$, and the fronts appear stable for all tested widths. 
Stability for $L>L_{\max}$, and ultimately on the whole plane $L=\infty$, remains an open problem.

Another interesting extension is to remove the viscosity terms from the second equation in \eqref{e:Rajopadhye}, thereby returning to Rajopadhye's original formulation \cite{Raj1996}. Our numerical experiments suggest that the fronts remain stable in this case.

\subsection*{Acknowledgement} The authors acknowledge support through NSF grant DMS-2407358. 

\bibliographystyle{amsplain}
\bibliography{Front}

\end{document}